\documentclass{amsart}

\newtheorem{theorem}{Theorem}

\theoremstyle{definition}

\theoremstyle{remark}
\newtheorem{remark}[theorem]{Remark}

\numberwithin{equation}{section}

\usepackage{hyperref}

\begin{document}
\title{Universality of Newton's method}

\author{A. G. Ramm}

\address{Mathematics Department, Kansas State University,
Manhattan, KS 66506-2602, USA}

\email{ramm@math.ksu.edu}



\date{}
\maketitle

\section{Introduction}

In many cases one is interested in solving operator equation
\begin{equation}
\label{eq1.1}
F(u) = h
\end{equation}
where $F$ is a nonlinear operator in real Hilbert space $H$. 
Let us assume that equation \eqref{eq1.1} has a 
solution $y$,
\begin{equation}
\label{eq1.2}
F(y) = f,
\end{equation}
that the Fr\'{e}chet derivative $F'(y)$ exists 
and is boundedly invertible:
\begin{equation}
\label{eq1.3}
\|[F'(y)]^{-1}\|\le m,\qquad m=const >0.
\end{equation}
Let us also assume that $F'(u)$ exists 
in the ball $B(y,R):=\{u:\|u-y\|\le R\}$, 
depends continuously on $u$, and $\omega(R)$ is its modulus of 
continuity in the ball $B(y,R)$:
\begin{equation}
\label{eq1.4}
\sup_{u,v\in B(y,R), \, \|u-v\|\le r} \|F'(u) - F'(v)\| = \omega(r). 
\end{equation}
The function $\omega(r)\ge 0$ is assumed to be continuous on the interval 
$[0, 2R]$, 
strictly increasing, and  $\omega(0)=0$.

A widely used method for solving equation \eqref{eq1.1} is the Newton
method: 
\begin{equation} \label{eq1.5} u_{n+1} = u_n -
[F'(u_n)]^{-1}F(u_n), \qquad u_0=z,   
\end{equation} 
where $z$ is an initial approximation.
Sufficient condition for the
convergence of the iterative scheme \eqref{eq1.5} to the solution $y$ of
equation \eqref{eq1.1} are proposed in \cite{D}, \cite{KA}, \cite{OR},
\cite{R499}, and references therein. These conditions in most cases
require a Lipschitz condition for $F'(u)$, a sufficient closeness of 
the
initial approximation $u_0$ to the solution $y$, and other conditions
( see, for example, \cite{D}, p.157.

In \cite{R499} a general method, the Dynamical Systems Method (DSM) is 
developed for solving equation \eqref{eq1.2}.

This method consists of finding a nonlinear operator $\Phi(t,u)$ such that 
the Cauchy problem
\begin{equation}
\label{eq1.6}
\dot{u} = \Phi(t,u),\qquad u(0) = u_0,
\end{equation}
has a unique global solution $u=u(t;u_0)$, there exists
$u(\infty) = \lim_{t\to\infty}u(t;u_0)$, and $F(u(\infty))=f$:
\begin{equation}
\label{eq1.7}
\exists ! u(t),\quad \forall t\ge 0;\quad \exists u(\infty);
\quad F(u(\infty)) = f.  
\end{equation}

Many examples of the possible choices of $\Phi(t,u)$ are given in \cite{R499}. 
Theoretical applications of the DSM are proposed in \cite{R489}, \cite{R575}. 
A particular choice of $\Phi$, namely, $\Phi = - [F'(u)]^{-1}(F(u)-h)$, 
leads to a continuous analog of the Newton method:
\begin{equation}
\label{eq1.8}
\dot{u}(t) = - [F'(u(t))]^{-1}(F(u(t))-h),\qquad u(0) = u_0; \qquad 
\dot{u}(t) = 
\frac{du(t)}{dt}.
\end{equation}

The question of general interest is: under what assumptions on $F,h$ and 
$u_0$, can one establish the conclusions \eqref{eq1.7}, that is, the 
global existence and uniqueness of the
solution to problem \eqref{eq1.8}, the existence of $u(\infty)$, 
and the relation $F(u(\infty))=h?$

The usual condition, sufficient for the local existence and uniqueness of 
the solution to \eqref{eq1.8} is the
local Lipschitz condition on the right-hand side of \eqref{eq1.8}. 
Such condition can be satisfied if $F'(u)$ satisfies a Lipschitz condition.

Our goal is to develop a novel approach to a study of equation \eqref{eq1.8}. 
This approach does not require a Lipschitz condition for $F'(u)$, and it 
leads to a justification of the conclusion \eqref{eq1.7} 
(with $h$ replacing $f$) for the solution to problem \eqref{eq1.8} 
under natural assumptions on $h$ and $u_0$. 

{\it Apparently for the first time a  proof of convergence of the 
continuous 
analog \eqref{eq1.8} of the Newton method and of the usual Newton 
method \eqref{eq1.5} is given 
without any smoothness assumptions on $F'(u)$, only the local continuity 
of $F'(u)$ is assumed, see \eqref{eq1.4}.} 

The Newton-type methods are widely used in theoretical, numerical and 
applied research, and
by this reason our results are of general interest for a wide audience.

Our results demonstrate the universality of the Newton method in the 
following sense: we prove that any operator equation \eqref{eq1.1}
can be solved by either the usual Newton method \eqref{eq1.5}
or by the DSM Newton method \eqref{eq1.8}, provided that
conditions \eqref{eq1.2}-\eqref{eq1.4} hold, the initial 
approximation $u_0$ is sufficiently close to $y$, where $y$ is the 
solution
of equation \eqref{eq1.2}, and the right-hand side $h$ 
in \eqref{eq1.1} is sufficiently close to $f$.
Precise formulation of the results is given in three Theorems.

The basic tool in this paper is a new version of the inverse function theorem. 
The novelty of this version is in a specification of the region in which 
the inverse function exists in terms of the modulus of continuity
of the operator $F'(u)$ in the ball $B(y,R)$. 

In Section~\ref{section2} we formulate and prove this version of the 
inverse function theorem. The result is stated as Theorem~\ref{theorem1}. 

In Section~\ref{section3} we justify the DSM for equation \eqref{eq1.8}. 
The result is stated in Theorem~\ref{theorem2}. 

In Section~\ref{section4} we prove convergence of the usual
Newton method \eqref{eq1.5}. The result is stated in 
Theorem~\ref{theorem3}.

\section{Inverse function theorem}
\label{section2}
Consider equation \eqref{eq1.1}. 

Let us make the following {\it Assumptions A):}
\begin{enumerate}
\item{}
Equation \eqref{eq1.2} and estimates \eqref{eq1.3}, \eqref{eq1.4} hold
in $B(y,R)$, 
\item{}
$h\in B(f,\rho),\qquad \rho = \frac{(1-q)R}{m}, \qquad q\in (0,1),$
\item{}
$m\omega(R) = q,  \qquad q\in (0,1)$.
\end{enumerate}

Assumption (3) defines $R$ uniquely because $\omega(r)$ is 
assumed to be strictly increasing. We assume that equation 
$m\omega(R) = q$ has a solution. This assumption is 
always satisfied if $q\in (0,1)$ is sufficiently small. The constant 
$m$ is defined in \eqref{eq1.3}.

Our first result, Theorem 1, says that under 
{\it Assumptions A)} equation \eqref{eq1.1} is uniquely solvable for any
$h$ in a sufficiently small neighborhood of $f$.

\begin{theorem}
\label{theorem1}
If  Assumptions A) hold then equation \eqref{eq1.1} has a unique 
solution 
$u$ for any $h\in B(f,\rho)$, and
\begin{equation}
\label{eq2.1}
\|[F'(u)]^{-1}\| \le \frac {m}{1-q},\qquad \forall u\in B(y,R). 
\end{equation}
\end{theorem}

\begin{proof}
Let us denote 
$$Q:=[F'(y)]^{-1}, \qquad \|Q\|\leq m.$$ 
Then equation 
\eqref{eq1.1} is equivalent to 
\begin{equation}
\label{eq2.2}
u =T(u), \qquad T(u):= u - Q(F(u) - h).
\end{equation}
Let us check that $T$ maps the ball $ B(y,R)$ into itself:
\begin{equation}
\label{eq2.3}
TB(y,R) \subset B(y,R),
\end{equation}
and that $T$ is a contraction mapping in this ball:
\begin{equation}
\label{eq2.4}
\|T(u) - T(v)\| \le q \|u-v\|, \qquad \forall u,v\in B(y,R),
\end{equation}
where $q\in (0,1)$ is defined in {\it Assumptions A)}.

If \eqref{eq2.2} and \eqref{eq2.3} are verified, then 
the contraction mapping principle 
guarantees existence and uniqueness of the solution to equation 
\eqref{eq2.2} in 
$B(y,R)$, where $R$ is defined by condition 3) in {\it Assumptions A)}. 

Let us check the inclusion \eqref{eq2.3}. One has 
\begin{equation}
\label{eq2.5}
J_1:= \|u - y - Q(F(u) - h)\| = \|u - y - Q[F(u) -F(y) +f - h]\|,
\end{equation}
and 
\begin{equation}
\label{eq2.6}
\begin{split}
F(u) - F(y) &= \int_0^1 F'(y+s(u-y)) ds (u-y) \\
&= F'(y)(u - y) 
+ \int_0^1 [F'(y+s(u-y)) - F'(y)]ds(u - y).
\end{split}
\end{equation}
Note that 
$$\|Q(f-h)\|\le m\rho,$$
and 
$$\sup_{s\in [0,1]}\|F'(y+s(u-y)) - F'(y)\|\le \omega(R).$$
Therefore, for any $u\in B(y,R)$ one gets from \eqref{eq1.3}, 
\eqref{eq2.4} and \eqref{eq2.5} the following estimate:
\begin{equation}
\label{eq2.7}
J_1 \le m\rho + m\omega(R)R \le (1-q)R+qR=R,
\end{equation}
where the inequalities
\begin{equation}
\label{eq2.8}
\|f-h\| \le \rho, \quad \|u-y\| \le R,
\end{equation}
and assumptions 2) and 3) in {\it Assumptions A)} were used.

Let us establish inequality \eqref{eq2.4}:
\begin{equation}
\label{eq2.9}
J_2 := \|T(u) - T(v)\| = \|u - v - Q(F(u) - F(v))\| 
\end{equation}

\begin{equation}
\label{eq2.10}
F(u) - F(v) = F'(y)(u - v) 
+ \int_0^1 [F'(v+s(u-v)) - F'(y)]ds(u - v).
\end{equation}
Note that 
$$\|v+s(u-v)-y\|= \|(1-s)(v-y)+s(u-y)\|\leq (1-s)R+sR=R.$$ 
Thus, from \eqref{eq2.9} and \eqref{eq2.10} one gets
\begin{equation}
\label{eq2.11}
J_2\le m \omega(R)\|u-v\| \le q\|u-v\|,\qquad \forall u,v\in 
B(y,R).
\end{equation}
Therefore, both conditions \eqref{eq2.3} and \eqref{eq2.4} are verified.
Consequently, the existence of the unique solution to \eqref{eq1.1}
in $B(y,R)$ is proved. \hfill $\Box$

Let us prove estimate \eqref{eq2.1}. One has
\begin{equation}
\label{eq2.12}
\begin{split}
[F'(u)]^{-1} & = [F'(y) + F'(u) - F'(y)]^{-1}\\
& = [I + (F'(y))^{-1}(F'(u) - F'(y))]^{-1}[F'(y)]^{-1},
\end{split}
\end{equation}
and
\begin{equation}
\label{eq2.13}
\|(F'(y))^{-1}(F'(u) - F'(y))\| \le m \omega(R) \le q,\qquad 
u\in B(y,R). 
\end{equation}
It is well known that if a linear operator $A$ satisfies the estimate 
$\|A\|\le q$,
where $q\in (0,1)$, then the inverse operator $(I+A)^{-1}$
does exist, and $\|(I+A)^{-1}|\le \frac 1 {1-q}$.
Thus, the operator $[I + (F'(y))^{-1}(F'(u) - F'(y))]^{-1}$ exists and
its norm can be estimated as follows: 
\begin{equation}
\label{eq2.14}
\|[I + (F'(y))^{-1}(F'(u) - F'(y))]^{-1}\| \le \frac {1}{1-q}.
\end{equation}
Consequently, \eqref{eq2.12} and \eqref{eq2.14} 
imply \eqref{eq2.1}.\hfill $\Box$ 

Theorem~\ref{theorem1} is proved. 
\end{proof}

\begin{remark}
\label{remark1}
If $h=h(t)\in C^1([0,T])$, then the solution $u=u(t)$ of equation \eqref{eq1.1} 
is $C^{1}([0,T])$ provided that Assumptions A) hold. 

Indeed, if $h=h(t)$, then a formal differentiation of equation 
\eqref{eq1.1} with respect to $t$ yields: 
\begin{equation}
\label{eq2.15}
F'(u(t))\dot{u}(t) = \dot{h}(t). 
\end{equation}
Since $u(t)\in B(y,R)$, the operator $F'(u(t))$ is boundedly invertible 
and depends continuously on 
$t$ because $u(t)$ does. Thus, 
$$\dot{u}(t) = [F'(u(t))]^{-1}\dot{h}(t),$$ 
so $\dot{u}(t)$ depends  on $t$ continuously.
 
The formal differentiation is justified if one proves that $u(t)$ 
is differentiable at any $t\in [0,T]$, that is,
\begin{equation}
\label{eq2.16}
u(t+k) - u(t) = A(t)k + o(k),\qquad k\to 0, \qquad t\in [0,T],
\end{equation}
where $A(t)\in H$ does not depend on $k$ and at the ends of the interval
$[0,T]$ the derivatives are understood as one-sided.

To establish relation \eqref{eq2.16} one uses equation \eqref{eq1.1}
and the assumption $h\in C^1([0,T])$. One has:
\begin{equation}
\label{eq2.17}
F(u(t+k)) - F(u(t)) = h(t+k) - h(t) = \dot{h}(t)k + o(k), \qquad k\to 0,
\end{equation}
and
\begin{equation}
\label{eq2.18}
F(u(t+k)) - F(u(t)) = \int_0^1 F'(u(t) + s(u(t+k)-u(t)))ds (u(t+k) - u(t)).
\end{equation}
The operator $\int_0^1 F'\bigg{(}u(t)+ s(u(t+k) - u(t))\bigg{)}ds$ 
is boundedly invertible (uniformly with respect to $k\in(0,k_0)$, 
where $0<k_0$ is a sufficiently small number) as long as 
$$\sup_{s\in[0,1]}\|u(t)+s(u(t+k)-u(t))-y\|\leq R,$$ 
see \eqref{eq2.1}. This inequality holds, as one can easily check:
$$\|u(t)+s(u(t+k)-u(t))-y\|=\|(1-s)(u(t)-y)+s(u(t+k)-y)\|\leq 
(1-s)R+sR=R.$$  
Therefore, \eqref{eq2.16} follows from \eqref{eq2.17} and \eqref{eq2.18}. 
Remark~\ref{remark1} is proved. \hfill $\Box$
\end{remark}

\section{Convergence of the DSM \eqref{eq1.8}}
\label{section3}

Consider the following equation 
\begin{equation}
\label{eq3.1}
F(u) = h + v(t),
\end{equation}
where
\begin{equation}
\label{eq3.2}
u = u(t),\quad v(t) = e^{-t} v_0,\quad v_0:= F(u_0) - h,\quad r = \|v_0\|.
\end{equation}
At $t=0$ equation \eqref{eq3.1} has a unique solution $u_0$. 

Let us make the following {\it Assumptions B)}:
\begin{enumerate}
\item{} {\it Assumptions A)} hold,
\item{} 
$h\in B(f,\delta),\qquad \delta +r\le \rho:=\frac{(1-q)R}{m}.$
\end{enumerate}

\begin{theorem}
\label{theorem2}
If Assumptions B) hold, then conclusions \eqref{eq1.7}, with $f$ replaced 
by $h$, hold for the solution 
of problem \eqref{eq1.8}.
\end{theorem}

\begin{proof} 

{\it 1. Proof of the global existence and uniqueness of the 
solution to problem \eqref{eq1.8}}.

One has 
$$\|h+v(t) - f\|\le\|h-f\|+\|v_0e^{-t}\|\le \delta+r\leq \rho, \qquad 
  \forall  t\ge 0.$$ 
Thus, it follows from Theorem~\ref{theorem1} 
that equation \eqref{eq3.1} has a unique solution 
$$u=u(t)\in B(y,R)$$
defined on the interval $t\in[0,\infty)$, and 
$u(t)\in C^1([0,\infty))$.
 
Differentiation of \eqref{eq3.1} with respect to $t$ yields 
\begin{equation}
\label{eq3.3}
F'(u)\dot{u} = \dot{v} = -v = - (F(u(t)) - h).
\end{equation}
Since $u(t)\in B(y,R)$, the operator $F'(u(t))$ is boundedly invertible, 
so equation \eqref{eq3.3} is 
equivalent to \eqref{eq1.8}. The initial condition $u(0)=u_0$ is 
satisfied, as was mentioned below \eqref{eq3.2}.
Therefore, the existence of the unique 
global solution to \eqref{eq1.8} is proved. \hfill $\Box$

{\it 2. Proof of the existence of $u(\infty)$.}

From \eqref{eq3.1}, \eqref{eq3.2},
 \eqref{eq2.1}, and \eqref{eq1.8} it follows that 
\begin{equation}
\label{eq3.4}
\|\dot{u}\| \le \frac{mr}{1-q}e^{-t}, \qquad q\in (0,1).
\end{equation}
This and the Cauchy criterion for the existence of the limit 
$u(\infty)$ imply that $u(\infty)$ exists. \hfill $\Box$

Integrating \eqref{eq3.4}, one gets 
\begin{equation}
\label{eq3.5}
\|u(t) - u_0\| \le \frac{mr}{1-q},
\end{equation}
and 
\begin{equation}
\label{eq3.6}
\|u(\infty) - u(t)\| \le \frac{mr}{1-q} e^{-t}.
\end{equation}

{\it 3. Proof of the relation $F(u(\infty)) = h$.}

Let us now prove that 
\begin{equation}
\label{eq3.7}
F(u(\infty)) = h.
\end{equation}
Relation \eqref{eq3.7} follows from \eqref{eq3.1} and \eqref{eq3.2} 
as $t\to\infty$, because $v(\infty) = 0$, $u(t)\in B(y,R)$, and $F$ is 
continuous in $B(y,R)$. \hfill  $\Box$
 
Theorem~\ref{theorem2} is proved. 
\end{proof}

\begin{remark}
\label{remark2}
Let us explain  why there is no assumption on the location of $u_0$
in Theorem~\ref{theorem2}. The reason is simple: in the proof 
of  Theorem~\ref{theorem2} it was established that $u(t)\in B(y,R)$
for all $t\geq 0$. Therefore, $u(0)\in B(y,R)$. 
\end{remark}

\section{The Newton method}
\label{section4}

The main goal in this Section is to prove convergence of the Newton method 
\begin{equation}
\label{eq4.1}
u_{n+1} = u_n - [F'(u_n)]^{-1}(F(u_n) - f),\qquad u_0 = z,
\end{equation}
to the solution $y$ of equation \eqref{eq1.2} 
{\it without any additional assumptions on the smoothness of $F'(u)$}.
By $z\in H$ we denote an initial approximation.

\begin{theorem}
\label{theorem3}
Assume that \eqref{eq1.2}--\eqref{eq1.4} and {\it Assumptions A)} hold, 
and that 
\begin{equation}
\label{eq4.2}
m\omega(R)=q\in (0, \frac 1 2), \qquad q_1\|z-y\|\leq R, \qquad q_1:=\frac 
{q}{1-q}.
\end{equation} 
Then process \eqref{eq4.1} converges to $y$.
\end{theorem}

\begin{proof}
One has
\begin{equation}
\label{eq4.3}
\begin{split}
u_{n+1} - y &= u_n - y - [F'(u_n)]^{-1}\int_0^1 F'(y+s(u_n-y))ds(u_n - 
y)\\
&= - [F'(u_n)]^{-1}\int_0^1 [F'(y+s(u_n-y)) - F'(u_n)]ds(u_n - y)
\end{split}
\end{equation}
Let 
$$a_n := \|u_n - y\|, \qquad a_0=\|z - y\|.$$ 
Then \eqref{eq4.2}, \eqref{eq4.3} and \eqref{eq2.1} imply
\begin{equation}
\label{eq4.4}
a_{n+1} \le \frac {m\omega(R)}{1-q}a_n \le \frac {q}{1-q} a_n:=q_1 a_n.
\end{equation}
From the assumption $q\in (0, \frac 1 2)$ one derives that $q_1\in (0,1)$. 
Thus, using \eqref{eq4.2}, one gets:
$$\|u_1 - y\|:=a_1\leq q_1a_0\leq R.$$ 
By induction one obtains:
$$\|u_n - y\|\le R, \qquad  \forall n=1,2,3,.....$$ 
Consequently, $u_n\in B(y,R)$ for all $n$, and estimates \eqref{eq2.1} and 
\eqref{eq4.4} are applicable for all $n$. 
Therefore, 
\eqref{eq4.4} implies
\begin{equation}
\label{eq4.5}
a_n \le q_1^{n-1} R, \qquad \forall n=1,2,3,.....  .
\end{equation}
Therefore, 
$$\lim_{n\to \infty}a_n=0.$$
Theorem~\ref{theorem3} is proved. 
\end{proof}

\end{document}